\documentclass[11pt,twoside]{amsart}
\usepackage{amsmath, amsthm, amscd, amsfonts, amssymb, graphicx, color, alltt}
\usepackage[bookmarksnumbered, plainpages]{hyperref}
\newtheorem{thm}{Theorem}[section]

\newtheorem{lem}[thm]{Lemma}

\newtheorem{conj}[thm]{Conjecture}

\numberwithin{equation}{section}
\usepackage[utf8]{inputenc}
\usepackage{pstricks-add}

\title{\bf Towards the Classification of Finite Simple Groups with exactly Three or Four Supercharacter Theories}

\author{\bf  A. R. Ashrafi and F. Koorepazan-Moftakhar}

\address{Department of Pure Mathematics, Faculty of Mathematical Sciences,
University of Kashan, Kashan 87317$-$53153, I. R. Iran}

\date{}
\begin{document}

\maketitle

\begin{abstract}
A supercharacter theory for a finite group $G$ is a set of
superclasses  each of which is a union of conjugacy classes
together with a set of sums of irreducible characters called
supercharacters that together satisfy certain compatibility
conditions. The aim of this paper is to give a description of some
finite simple groups with exactly three or four supercharacter
theories.

\vskip 3mm


\vskip 3mm

\end{abstract}

\bigskip

\section{Introduction}
Throughout this paper $G$ is a finite group,  $d(n)$ denotes the
number of divisors of positive integer $n$, $Irr(G)$ stands for
the set of all ordinary irreducible character of $G$ and $Con(G)$
is the set of all conjugacy classes of $G$. For other notations
and terminology concerning character theory, we refer to the
famous book of Isaacs \cite{11}. Following Diaconis and Isaacs
\cite{7}, a pair ($\mathcal{X}$, $\mathcal{K}$) together with the
choices of characters $\chi_X$ is called a \textbf{supercharacter
theory} of $G$ if the following conditions are satisfied:
\begin{enumerate}
\item  $\mathcal{X}$ is a partition of $Irr(G)$ and  $\mathcal{K}$ is a partition of $Con(G)$;

\item $\{ 1\} \in \mathcal{K}$;

\item the characters $\chi_X$, $X \in \mathcal{X}$, are constant on the members of $\mathcal{K}$;

\item $|\mathcal{X}|$ = $|\mathcal{K}|$.
\end{enumerate}
The elements of  $\mathcal{X}$ and $\mathcal{K}$  are called
\textbf{supercharacters} and \textbf{superclasses} of $G$,
respectively. It is easy to see that $m(G) = (Irr(G),Con(G))$ and
$M(G) = (\mathcal{X}, \mathcal{K})$, where $\mathcal{X}$ = $\{ \{
1\},  Irr(G) \setminus \{ 1\}\}$ and $\mathcal{K}$ = $\{
\{1\},Con(G) \setminus\{ 1\}\}$ are supercharacter theories of
$G$ which are called the trivial supercharacter theories of $G$.
The set of all supercharacter theories of a finite group $G$ is
denoted by $Sup(G)$ and set $s(G) = |Sup(G)|$. These notions were
first introduced by  Andre for finite unitriangular groups using
polynomial equations defining certain algebraic varieties
\cite{0,1,2}.

Following Hendrickson \cite{9}, we assume that $Part(S)$ denotes
the set of all partitions of a set $S$. If $\mathcal{X}$ and
$\mathcal{Y}$ are two elements of $Part(S)$ then we say that
$\mathcal{X}$ is a \textbf{refinement} of $\mathcal{Y}$ or
$\mathcal{Y}$ is \textbf{coarser than} $\mathcal{X}$ and we
write  ``$\mathcal{X}$ $\preceq$ $\mathcal{Y}$", if
$[a]_{\mathcal{X}} \subseteq [a]_{\mathcal{Y}}$ for all $a \in
S$. For two supercharacter theories $(\mathcal{X},\mathcal{K})$
and $(\mathcal{Y},\mathcal{L})$, we define
$(\mathcal{X},\mathcal{K}) \vee (\mathcal{Y},\mathcal{L})$ =
$(\mathcal{X} \vee \mathcal{Y}, \mathcal{K} \vee \mathcal{L})$.
It is well-know that $(Part(S),\preceq)$ is a lattice which is
called the \textbf{partition lattice} of $S$. By
\cite[Proposition 2.16]{10}, the join  of two supercharacter
theories of a group $G$ is again a supercharacter theory for $G$,
but it is possible to find a pair of supercharacter theories such
that their meet is not a supercharacter theory. This shows that
the set of all supercharacter theories of a finite group with
usual join and meet don't constitute a lattice in general.

Burkett et al. \cite{6} gave a classification of finite groups
with exactly two supercharacter theories. They proved that a
finite group $G$ has exactly two supercharacter theories if and
only if $G$ is isomorphic to the cyclic group $Z_3$, the
symmetric group $S_3$ or the simple group $Sp(6, 2)$. The aim of
this paper is to continue this work towards a classification of
finite simple groups with exactly three or four supercharacter
theories.

If $p$ and $2p + 1$ are primes then $p$ is called a
\textbf{Sophie Germain prime} and $2p + 1$ is said to be a
\textbf{safe prime}. The safe primes are recorded in on-line
encyclopedia of integer sequences as A005385, see \cite{13} for
details. The first few members of this sequence is $5,$ $7$,
$11$, $23$, $47$, $59$, $83$, $107$, $167$, $179$, $227$, $263$,
$347$, $359$, $383$, $467$, $479$, $503$, $563$, $587$, $719$,
$839$, $863$, $887$, $983$, $1019$, $1187$, $1283$, $1307$,
$1319$, $1367$, $1439$, $1487$, $1523$, $1619$, $1823$, $1907$.
These two sequences of prime numbers have several applications in
public key cryptography and primality testing and it has been
conjectured that there are infinitely many Sophie Germain primes,
but this remains unproven \cite{65}.

\section{Supercharacter Theory Construction for Sporadic  Groups}

The aim of this section is to compute some supercharacter
theories for sporadic simple groups. We use the following simple
lemma in our calculations:

\begin{lem}\label{11}
Let $G$ be a finite group, $\chi$ be a non-real valued irreducible
character of $G$ and $x \in G$ such that $\chi(x)$ is a non-real
number. We also assume that each row and column of $G$ has at most
two non-real numbers. Define:
\begin{eqnarray*}
\mathcal{X} &=& \{\{\chi_{1}\}, \{\chi, \overline{\chi}\}, Irr(G) - \{\chi_{1}, \chi, \overline{\chi}\}\}\\
\mathcal{K} &=& \{\{e\}, \{x^{G}, (x^{-1})^{G}\}, Con(G) - \{e,
x^{G}, (x^{-1})^{G}\}\}.
\end{eqnarray*}
Then {$ (\mathcal{X}, \mathcal{K}) $} is a supercharacter theory
for $G$.
\end{lem}

\begin{proof}
To prove, it is enough to investigate the main condition of
supercharacter theory for the conjugacy classes $x^{G}$, $
(x^{-1})^{G} $ and irreducible characters $\chi$,
$\overline{\chi}$. By definition of ${\sigma}_{X}$ on  $ X =
\{x^{G}, (x^{-1})^{G}\} $,
\begin{eqnarray*}
\sigma_{X}(x^{-1}) &=& \chi(1)\chi(x^{-1}) + \overline{\chi}(1)\overline{\chi}(x^{-1})\\
&=& \chi(1) \overline{\chi(x)} + \chi(1)\chi(x)\\
&=& \overline{\chi(1)}~\overline{\chi(x)} + \chi(1)\chi(x)\\
&=& \sigma_{X}(x).
\end{eqnarray*}
Therefore, {$ \sigma_{X} $} is constant on the part {$ \{x^{G},
(x^{-1})^{G}\} $} of $\mathcal{K}$, as desired.
\end{proof}

The following lemma is important for constructing supercharacter theories on simple groups.

\begin{lem}\label{12}
Suppose $G$ is  a finite group, $A = \{ \chi(x) \ | \ \chi \in Irr(G)\}$ and $\mathbb{Q}(A)$ denotes the filed generated by $\mathbb{Q}$ and $A$. Then the following holds:
\begin{enumerate}
\item If $\mathcal{X}(G) = \{ \{ \chi, \overline{\chi} \} \ | \ \chi \in  Irr(G) \}$ and $\mathcal{K}(G) = \{ \{ x^G, (x^{-1})^G\} \ | \ x \in G\}$
then $(\mathcal{X}, \mathcal{K})$ is a supercharacter theory of $G$.

\item If $\Gamma = Gal(\frac{\mathbb{Q(A)}}{\mathbb{Q}})$, $\mathcal{X}(G)$ is the set of all orbits of $\Gamma$ on $Irr(G)$ and $\mathcal{K}(G)$ is
the set of all orbits of $\Gamma$ on $Con(G)$ then  $(\mathcal{X},
\mathcal{K})$ is a supercharacter theory of $G$.

\end{enumerate}
\end{lem}

\begin{proof}
The proof  follows from \cite[p. 2360]{7} and \cite{12}.
\end{proof}

To calculate the supercharacter of a finite group $G$, we first
sort the character table of $G$ by the following GAP commands \cite{19}:
\begin{alltt}
u:=CharacterTable(G);
t:=CharacterTableWithSortedCharacters(u);
\end{alltt}
Then we prepare a GAP program to check whether or not a given pair
$(\mathcal{X},\mathcal{K})$  of a partition $\mathcal{K}$ for
conjugacy classes and another partition $\mathcal{X}$ for
irreducible characters constitutes a supercharacter theory. To
find this pair of partitions, we usually apply Lemma \ref{11}.

\begin{thm}
The Mathieu groups $M_{11}, M_{12}, M_{22}, M_{23}$ and $M_{24}$ have at least five supercharacter theories.
\end{thm}

\begin{proof}
We will present five supercharacter theories for each Mathieu
group as follows:
\begin{itemize}
\item \textit{The Mathieu group $M_{11}$}. Suppose the irreducible characters and conjugacy classes of the Mathieu group $M_{11}$ are
$Irr(M_{11})$ = $\{  \chi_1, \chi_2, \ldots, \chi_{10}\}$ and
$Con(M_{11})$ = $\{  x_1^{M_{11}}, x_2^{M_{11}}, \ldots,
x_{10}^{M_{11}}\}$, respectively. We now define:
\begin{small}
\begin{eqnarray*}
\mathcal{K}_{1} &=& \left\lbrace \{1\}, \{x_{i}^{M_{11}}\} (2 \leq i \leq 8), \{x_{9}^{M_{11}},x_{10}^{M_{11}}\} \right\rbrace, \\
\mathcal{X}_{1} &=& \left\lbrace \{1\}, \{\chi_{i}\} (2 \leq i
\leq 5), \{\chi_{6}, \chi_{7}\},\{\chi_{i}\} (8 \leq i \leq
10)\right\rbrace, \\
\mathcal{K}_{2} &=& \left\lbrace \{1\}, \{x_{i}^{M_{11}}\} (2 \leq i \leq 6), \{x_{7}^{M_{11}},x_{8}^{M_{11}}\}, \{x_{i}^{M_{11}}\} (9 \leq i \leq 10) \right\rbrace, \\
\mathcal{X}_{2} &=& \left\lbrace \{1\}, \{\chi_{2}\}, \{\chi_{3},
\chi_{4}\},\{\chi_{i}\} (5 \leq i \leq 10)\right\rbrace.
\end{eqnarray*}
\end{small}
 By Lemma \ref{11} and \cite[Proposition 2.16]{10},
$\mathcal{C}_{1} = (\mathcal{X}_{1}, \mathcal{K}_{1})$,
$\mathcal{C}_{2} = (\mathcal{X}_{2},\mathcal{K}_{2})$ and
$\mathcal{C}_{3} = \mathcal{C}_{1} \vee \mathcal{C}_{2}$ are
supercharacter theories of $M_{11}$. Since $\mathcal{C}_{1}$,
$\mathcal{C}_{2}$, $\mathcal{C}_{3}$, $m(M_{11})$ and $M(M_{11})$
are distinct, $s(M_{11}) \geq 5$.

\item \textit{The Mathieu group $M_{12}$}. We assume that the irreducible characters and conjugacy classes of the Mathieu group $M_{12}$
are $Irr(M_{12})$ = $\{  \chi_1, \chi_2, \ldots, \chi_{15}\}$ and
$Con(M_{12})$ = $\{  x_1^{M_{12}},  \ldots, x_{15}^{M_{12}} \}$,
respectively. Define:
\begin{footnotesize}
\begin{eqnarray*}
\mathcal{K}_{1} &=& \left\lbrace \{1\}, \{x_{i}^{M_{12}}\} (2 \leq i \leq 13), \{x_{14}^{M_{12}},x_{15}^{M_{12}}\} \right\rbrace, \\
\mathcal{X}_{1} &=& \left\lbrace \{1\}, \{\chi_{i}\} (2 \leq i
\leq 3), \{\chi_{4}, \chi_{5}\},\{\chi_{i}\} (6 \leq i \leq
15)\right\rbrace, \\
\mathcal{K}_{2} &=& \left\lbrace \{1\},
\{x_{i}^{M_{12}}\} (2 \leq i \leq 5),
\{x_{6}^{M_{12}},x_{7}^{M_{12}}\}, \{x_{i}^{M_{12}}\} (8 \leq i
\leq 10),
\{x_{11}^{M_{12}},x_{12}^{M_{12}}\}, \{x_{i}^{M_{12}}\} (13 \leq i \leq 15) \right\rbrace, \\
\mathcal{X}_{2} &=& \left\lbrace \{1\}, \{\chi_{2}, \chi_{3}\},
\{\chi_{i}\} (4 \leq i \leq 8), \{\chi_{9},
\chi_{10}\},\{\chi_{i}\} (11 \leq i \leq 15)\right\rbrace .
\end{eqnarray*}
\end{footnotesize}
Then by Lemma \ref{11} and \cite[Proposition 2.16]{10}, the pairs
$\mathcal{C}_{1} = (\mathcal{X}_{1}, \mathcal{K}_{1})$,
$\mathcal{C}_{2} = (\mathcal{X}_{2},\mathcal{K}_{2})$ and
$\mathcal{C}_{3} = \mathcal{C}_{1} \vee \mathcal{C}_{2}$ are
three supercharacter theories of $M_{12}$ different from
$m(M_{12})$ and $M(M_{12})$. This proves that $s(M_{12}) \geq 5$,
as desired.

\item \textit{The Mathieu group $M_{22}$}. Suppose $Irr(M_{22})$ = $\{  \chi_1,  \ldots,
\chi_{12}\}$ and $Con(M_{22})$ = $\{ x_1^{M_{22}},  \ldots,
x_{12}^{M_{22}}\}$. We define:
\begin{small}
\begin{eqnarray*}
\mathcal{K}_{1} &=& \left\lbrace \{1\}, \{x_{i}^{M_{22}}\} (2 \leq i \leq 10), \{x_{11}^{M_{22}},x_{12}^{M_{22}}\} \right\rbrace, \\
\mathcal{X}_{2} &=& \left\lbrace \{1\}, \{\chi_{i}\} (2 \leq i
\leq 9), \{\chi_{10}, \chi_{11}\},\{\chi_{12}\} \right\rbrace, \\
\mathcal{K}_{2} &=& \left\lbrace \{1\}, \{x_{i}^{M_{22}}\} (2 \leq i \leq 7), \{x_{8}^{M_{22}},x_{9}^{M_{22}}\}, \{x_{i}^{M_{22}}\} (10 \leq i \leq 12) \right\rbrace, \\
\mathcal{X}_{2} &=& \left\lbrace \{1\}, \{\chi_{2}\}, \{\chi_{3},
\chi_{4}\},\{\chi_{i}\} (5 \leq i \leq 12)\right\rbrace.
\end{eqnarray*}
\end{small}
Then by Lemma \ref{11}  the pairs $\mathcal{C}_{1} =
(\mathcal{X}_{1}, \mathcal{K}_{1})$ and $\mathcal{C}_{2} =
(\mathcal{X}_{2},\mathcal{K}_{2})$ are supercharacter theories of
$M_{22}$. We now apply \cite[Proposition 2.16]{10} to prove that
$\mathcal{C}_{3} = \mathcal{C}_{1} \vee \mathcal{C}_{2}$ is
another supercharacter theory for $M_{22}$ which shows that
$s(M_{22}) \geq 5$.

\item \textit{The Mathieu group $M_{23}$}. Suppose $Irr(M_{23})$ = $\{  \chi_1,  \ldots, \chi_{17}\}$ and
$Con(M_{23})$ = $\{  x_1^{M_{23}},  \ldots, x_{17}^{M_{23}}\}$
are the irreducible characters and  conjugacy classes of the
Mathieu group $M_{23}$, respectively. If we can present three
supercharacter theories for $M_{23}$ different from $m(M_{23})$
and $M(M_{23})$ then it can be easily proved that $s(M_{23}) \geq
5$, as desired. Define:
\begin{small}
\begin{eqnarray*}
\mathcal{K}_{1} &=& \left\lbrace \{1\}, \{x_{i}^{M_{23}}\} (2 \leq i \leq 15), \{x_{16}^{M_{23}},x_{17}^{M_{23}}\} \right\rbrace, \\
\mathcal{X}_{1} &=& \left\lbrace \{1\}, \{\chi_{i}\} (2 \leq i
\leq 9), \{\chi_{10}, \chi_{11}\}, \{\chi_{i}\} (12 \leq i \leq
17) \right\rbrace, \\
\mathcal{K}_{2} &=& \left\lbrace \{1\}, \{x_{i}^{M_{23}}\} (2 \leq i \leq 9), \{x_{10}^{M_{23}},x_{11}^{M_{23}}\}, \{x_{i}^{M_{23}}\} (12 \leq i \leq 17) \right\rbrace, \\
\mathcal{X}_{2} &=& \left\lbrace \{1\}, \{\chi_{i}\} (2 \leq i
\leq 11), \{\chi_{12}, \chi_{13}\}, \{\chi_{i}\} (14 \leq i \leq
17) \right\rbrace.
\end{eqnarray*}
\end{small}
To complete the proof, it is enough to apply  Lemma \ref{11} and
\cite[Proposition 2.16]{10} for proving that $\mathcal{C}_{1} =
(\mathcal{X}_{1}, \mathcal{K}_{1})$, $\mathcal{C}_{2} =
(\mathcal{X}_{2},\mathcal{K}_{2})$ and $\mathcal{C}_{3} =
\mathcal{C}_{1} \vee \mathcal{C}_{2}$ are supercharacter theories
of $M_{23}$.

\item \textit{The Mathieu group $M_{24}$}. We now assume that $Irr(M_{24})$ = $\{  \chi_1, \chi_2, \ldots, \chi_{26}\}$ and
$Con(M_{24})$ = $\{  x_1^{M_{24}}, x_2^{M_{24}}, \ldots,
x_{26}^{M_{24}}\}$. If we define:
\begin{small}
\begin{eqnarray*}
\mathcal{K}_{1} &=& \left\lbrace \{1\}, \{x_{i}^{M_{24}}\} (2 \leq i \leq 20), \{x_{21}^{M_{24}},x_{22}^{M_{24}}\}, \{x_{i}^{M_{24}}\} (23 \leq i \leq 26) \right\rbrace, \\
\mathcal{X}_{1} &=& \left\lbrace \{1\}, \{\chi_{i}\} (2 \leq i
\leq 4), \{\chi_{5}, \chi_{6}\}, \{\chi_{i}\} (7 \leq i \leq 26)
\right\rbrace, \\
\mathcal{K}_{2} &=& \left\lbrace \{1\}, \{x_{i}^{M_{24}}\} (2 \leq i \leq 24), \{x_{25}^{M_{24}},x_{26}^{M_{24}}\} \right\rbrace, \\
\mathcal{X}_{2} &=& \left\lbrace \{1\}, \{\chi_{i}\} (2 \leq i
\leq 9), \{\chi_{10}, \chi_{11}\}, \{\chi_{i}\} (12 \leq i \leq
26) \right\rbrace.
\end{eqnarray*}
\end{small}
then by a similar calculation as other cases, we can prove that
 $\mathcal{C}_{1} = (\mathcal{X}_{1},
\mathcal{K}_{1})$, $\mathcal{C}_{2} =
(\mathcal{X}_{2},\mathcal{K}_{2})$ and  $\mathcal{C}_{3} =
\mathcal{C}_{1} \vee \mathcal{C}_{2}$ are three supercharacter
theories different from $m(M_{24})$ and $M(M_{24})$, proving this
case.
\end{itemize}
This completes the proof.
\end{proof}

\begin{thm}
The Leech lattice groups have at least five supercharacter theories.
\end{thm}

\begin{proof}
There are seven Leech lattice simple groups. These are $HS$,
$J_2$, $Co_1, Co_2, Co_3$, $McL$ and $Suz$. Our main proof will
consider seven cases as follows:
\begin{itemize}
\item \textit{The Higman-Sims group HS}. To establish five supercharacter theories for $HS$, we assume that
$Irr(HS) = \{  \chi_i\}_{1 \leq i \leq 24}$ and $Con(HS) = \{
x_i^{HS}\}_{1 \leq i \leq 24}$. Define:
\begin{small}
\begin{eqnarray*}
\mathcal{K}_{1} &=& \left\lbrace \{1\}, \{x_{i}^{HS}\} (2 \leq i \leq 22), \{x_{23}^{HS},x_{24}^{HS}\} \right\rbrace, \\
\mathcal{X}_{1} &=& \left\lbrace \{1\}, \{\chi_{i}\} (2 \leq i \leq 10), \{\chi_{11}, \chi_{12}\}, \{\chi_{i}\} (13 \leq i \leq 24) \right\rbrace, \\
\mathcal{K}_{2} &=& \left\lbrace \{1\}, \{x_{i}^{HS}\} (2 \leq i \leq 18), \{x_{19}^{HS},x_{20}^{HS}\}, \{x_{i}^{HS}\} (21 \leq i \leq 24) \right\rbrace, \\
\mathcal{X}_{2} &=& \left\lbrace \{1\}, \{\chi_{i}\} (2 \leq i
\leq 13), \{\chi_{14}, \chi_{15}\}, \{\chi_{i}\} (16 \leq i \leq
24) \right\rbrace.
\end{eqnarray*}
\end{small}
Then by Lemma \ref{11}, one can see that $\mathcal{C}_{1} =
(\mathcal{X}_{1}, \mathcal{K}_{1})$ and $\mathcal{C}_{2} =
(\mathcal{X}_{2},\mathcal{K}_{2})$ are supercharacter theories of
$HS$. We now apply \cite[Proposition 2.16]{10}, to prove that
$\mathcal{C}_{3} = \mathcal{C}_{1} \vee \mathcal{C}_{2}$ is
another supercharacter theory of $HS$. These supercharacter
theories are different from $m(HS)$ and $M(HS)$ which concludes
that $s(HS) \geq 5$.

\item \textit{The  Conway group $Co_{1}$}. Suppose $Irr(Co_{1}) = \{  \chi_i\}_{1 \leq i \leq 101}$ and $Con(Co_{1}) = \{  x_i^{Co_{1}}\}_{1 \leq i \leq 101}$.
Define:
\begin{small}
\begin{eqnarray*}
\mathcal{K}_{1} &=& \left\lbrace \{1\}, \{x_{i}^{Co_{1}}\} (2 \leq i \leq 96), \{x_{97}^{Co_{1}},x_{98}^{Co_{1}}\}, \{x_{i}^{Co_{1}}\} (99 \leq i \leq 101) \right\rbrace, \\
\mathcal{X}_{1} &=& \left\lbrace \{1\}, \{\chi_{i}\} (2 \leq i
\leq 26), \{\chi_{27}, \chi_{28}\}, \{\chi_{i}\} (29 \leq i \leq
101) \right\rbrace, \\
\mathcal{K}_{2} &=& \left\lbrace \{1\}, \{x_{i}^{Co_{1}}\} (2 \leq i \leq 77), \{x_{78}^{Co_{1}},x_{79}^{Co_{1}}\}, \{x_{i}^{Co_{1}}\} (80 \leq i \leq 101) \right\rbrace, \\
\mathcal{X}_{2} &=& \left\lbrace \{1\}, \{\chi_{i}\} (2 \leq i
\leq 16), \{\chi_{17}, \chi_{18}\}, \{\chi_{i}\} (19 \leq i \leq
101) \right\rbrace.
\end{eqnarray*}
\end{small}
By Lemma \ref{11}, $\mathcal{C}_{1} = (\mathcal{X}_{1},
\mathcal{K}_{1})$ and $\mathcal{C}_{2} =
(\mathcal{X}_{2},\mathcal{K}_{2})$ are supercharacter theories of
$Co_1$ and by \cite[Proposition 2.16]{10}, $\mathcal{C}_{3} =
\mathcal{C}_{1} \vee \mathcal{C}_{2}$. Since these supercharacter
theories are different from  $m(Co_{1})$ and $M(Co_{1})$,
$s(Co_1) \geq 5$. Hence the result  follows.

\item \textit{The second Conway group $Co_{2}$}. Let
$Irr(Co_{2}) = \{  \chi_i\}_{1 \leq i \leq 60}$ and $Con(Co_{2})
= \{  x_i^{Co_{2}}\}_{1 \leq i \leq 60}$. Define:
{\small\begin{eqnarray*} \mathcal{K}_{1} &=& \left\lbrace \{1\},
\{x_{i}^{Co_{2}}\} (2 \leq i \leq 45),
\{x_{46}^{Co_{2}},x_{47}^{Co_{2}}\}, \{x_{i}^{Co_{2}}\} (48 \leq
i \leq 58),
\{x_{59}^{Co_{2}},x_{60}^{Co_{2}}\}\right\rbrace, \\
\mathcal{X}_{1} &=& \left\lbrace \{1\}, \{\chi_{i}\} (2 \leq i
\leq 11), \{\chi_{12}, \chi_{13}\}, \{\chi_{i}\} (14 \leq i \leq
30), \{\chi_{31}, \chi_{32}\}, \{\chi_{i}\} (33 \leq i \leq 60)
\right\rbrace, \\
\mathcal{K}_{2} &=& \left\lbrace \{1\}, \{x_{i}^{Co_{2}}\} (2 \leq i \leq 52), \{x_{53}^{Co_{2}},x_{54}^{Co_{2}}\}, \{x_{i}^{Co_{2}}\} (55 \leq i \leq 60)\right\rbrace, \\
\mathcal{X}_{2} &=& \left\lbrace \{1\}, \{\chi_{i}\} (2 \leq i
\leq 9), \{\chi_{10}, \chi_{11}\}, \{\chi_{i}\} (12 \leq i \leq
60) \right\rbrace.
\end{eqnarray*}}
By Lemma \ref{11} and  \cite[Proposition 2.16]{10},
$\mathcal{C}_{1} = (\mathcal{X}_{1}, \mathcal{K}_{1})$,
$\mathcal{C}_{2} = (\mathcal{X}_{2},\mathcal{K}_{2})$ and
$\mathcal{C}_{3} = \mathcal{C}_{1} \vee \mathcal{C}_{2}$ are
supercharacter theories of the Conway group $Co_2$. Since these
supercharacter theories are different from $m(Co_{2})$ and
$M(Co_{2})$, $s(Co_2) \geq 5$.

\item \textit{The third Conway group $Co_{3}$}. We assume that
$Irr(Co_{3})$ = $\{  \chi_i\}_{1 \leq i \leq 42}$ and
$Con(Co_{3})$ = $\{  x_i^{Co_{3}}\}_{1 \leq i \leq 42}$. Define:
\begin{footnotesize}
\begin{eqnarray*} \mathcal{K}_{1} &=& \left\lbrace \{1\},
\{x_{i}^{Co_{3}}\} (2 \leq i \leq 23),
\{x_{24}^{Co_{3}},x_{25}^{Co_{3}}\}, \{x_{i}^{Co_{3}}\} (26 \leq
i \leq 35),
\{x_{36}^{Co_{3}},x_{37}^{Co_{3}}\}, \{x_{i}^{Co_{3}}\} (38 \leq i \leq 42)\right\rbrace, \\
\mathcal{X}_{1} &=& \left\lbrace \{1\}, \{\chi_{i}\} (2 \leq i
\leq 5), \{\chi_{6}, \chi_{7}\}, \{\chi_{i}\} (8 \leq i \leq 17),
\{\chi_{18}, \chi_{19}\}, \{\chi_{i}\} (20 \leq i \leq 42)
\right\rbrace,\\
\mathcal{K}_{2} &=& \left\lbrace \{1\}, \{x_{i}^{Co_{3}}\} (2 \leq i \leq 37), \{x_{38}^{Co_{3}},x_{39}^{Co_{3}}\}, \{x_{i}^{Co_{3}}\} (40 \leq i \leq 42)\right\rbrace, \\
\mathcal{X}_{2} &=& \left\lbrace \{1\}, \{\chi_{i}\} (2 \leq i
\leq 15), \{\chi_{16}, \chi_{17}\}, \{\chi_{i}\} (18 \leq i \leq
42) \right\rbrace.
\end{eqnarray*}
\end{footnotesize}
Then $m(Co_{3})$, $M(Co_{3})$, $\mathcal{C}_{1} =
(\mathcal{X}_{1}, \mathcal{K}_{1})$, $\mathcal{C}_{2} =
(\mathcal{X}_{2},\mathcal{K}_{2})$ and $\mathcal{C}_{3} =
\mathcal{C}_{1} \vee \mathcal{C}_{2}$ are five supercharacter
theories for $Co_3$. Thus $s(Co_3) \geq 5$, as required.

\item \textit{The McLaughlin group $McL$}. Suppose $Irr(McL) = \{  \chi_i\}_{1 \leq i \leq 24}$ and
$Con(McL) = \{  x_i^{McL}\}_{1 \leq i \leq 24}$. We now define:
\begin{small}
\begin{eqnarray*}
\mathcal{K}_{1} &=& \left\lbrace \{1\}, \{x_{i}^{McL}\} (2 \leq i \leq 15), \{x_{16}^{McL},x_{17}^{McL}\}, \{x_{i}^{McL}\} (18 \leq i \leq 24) \right\rbrace, \\
\mathcal{X}_{1} &=& \left\lbrace \{1\}, \{\chi_{i}\} (2 \leq i
\leq 6), \{\chi_{7}, \chi_{8}\}, \{\chi_{i}\} (9 \leq i \leq 24)
\right\rbrace, \\
\mathcal{K}_{2} &=& \left\lbrace \{1\}, \{x_{i}^{McL}\} (2 \leq i \leq 12), \{x_{13}^{McL},x_{14}^{McL}\}, \{x_{i}^{McL}\} (15 \leq i \leq 24) \right\rbrace, \\
\mathcal{X}_{2} &=& \left\lbrace \{1\}, \{\chi_{i}\} (2 \leq i
\leq 20), \{\chi_{21}, \chi_{22}\}, \{\chi_{i}\} (23 \leq i \leq
24) \right\rbrace.
\end{eqnarray*}
\end{small}
Since $m(McL)$, $M(McL)$, $\mathcal{C}_{1} = (\mathcal{X}_{1},
\mathcal{K}_{1})$,  $\mathcal{C}_{2} =
(\mathcal{X}_{2},\mathcal{K}_{2})$ and $\mathcal{C}_{3} =
\mathcal{C}_{1} \vee \mathcal{C}_{2}$ are five supercharacter
theories of $McL$, $s(McL) \geq 5$.

\item {\textit{The Suzuki group $Suz$}}. The irreducible characters and conjugacy classes for the Suzuki  group $Suz$ are
$Irr(Suz) = \{  \chi_i\}_{1 \leq i \leq 43}$ and $Con(Suz) = \{
x_i^{Suz}\}_{1 \leq i \leq 43}$, respectively. Define:
\begin{small}
\begin{eqnarray*}
\mathcal{K}_{1} &=& \left\lbrace \{1\}, \{x_{i}^{Suz}\} (2 \leq i \leq 40), \{x_{41}^{Suz},x_{42}^{Suz}\}, \{x_{43}^{Suz}\} \right\rbrace, \\
\mathcal{X}_{1} &=& \left\lbrace \{1\}, \{\chi_{i}\} (2 \leq i
\leq 24), \{\chi_{25}, \chi_{26}\}, \{\chi_{i}\} (27 \leq i \leq
43) \right\rbrace, \\
\mathcal{K}_{2} &=& \left\lbrace \{1\}, \{x_{i}^{Suz}\} (2 \leq i \leq 34), \{x_{35}^{Suz},x_{36}^{Suz}\}, \{x_{i}^{Suz}\} (37 \leq i \leq 43) \right\rbrace, \\
\mathcal{X}_{2} &=& \left\lbrace \{1\}, \{\chi_{i}\} (2 \leq i
\leq 12), \{\chi_{13}, \chi_{14}\}, \{\chi_{i}\} (15 \leq i \leq
43) \right\rbrace.
\end{eqnarray*}
\end{small}
Since this group has five supercharacter theories $m(Suz)$,
$M(Suz)$, $\mathcal{C}_{1} = (\mathcal{X}_{1}, \mathcal{K}_{1})$,
$\mathcal{C}_{2} = (\mathcal{X}_{2},\mathcal{K}_{2})$ and
$\mathcal{C}_{3} = \mathcal{C}_{1} \vee \mathcal{C}_{2}$, we
conclude  that $s(Suz) \geq 5$.
\end{itemize}
This completes the proof.
\end{proof}

\begin{thm}
The Monster sections have at least five supercharacter theories.
\end{thm}

\begin{proof}
The Monster sections are eight simple groups $He$, $HN$, $Th$,
$Fi_{22}$, $Fi_{23}$, $Fi_{24}^\prime$, $B$ and $M$. We will
present five supercharacter theories in each case  as follows:
\begin{itemize}
\item {\textit{The Held group $He$}}. Suppose $Irr(He) = \{  \chi_i\}_{1 \leq i \leq 33}$ and
$Con(He) = \{  x_i^{He}\}_{1 \leq i \leq 33}$ are irreducible
characters and  conjugacy classes of the group $He$,
respectively. Define:
\begin{small}
\begin{eqnarray*}
\mathcal{K}_{1} &=& \left\lbrace \{1\}, \{x_{i}^{He}\} (2 \leq i \leq 27), \{x_{28}^{He},x_{29}^{He}\}, \{x_{i}^{He}\} (30 \leq i \leq 33) \right\rbrace, \\
\mathcal{X}_{1} &=& \left\lbrace \{1\}, \{\chi_{i}\} (2 \leq i
\leq 29), \{\chi_{30}, \chi_{31}\}, \{\chi_{i}\} (32 \leq i \leq
33) \right\rbrace, \\
\mathcal{K}_{2} &=& \left\lbrace \{1\}, \{x_{i}^{He}\} (2 \leq i \leq 25), \{x_{26}^{He},x_{27}^{He}\}, \{x_{i}^{He}\} (28 \leq i \leq 33) \right\rbrace, \\
\mathcal{X}_{2} &=& \left\lbrace \{1\}, \{\chi_{i}\} (2 \leq i
\leq 6), \{\chi_{7}, \chi_{8}\}, \{\chi_{i}\} (9 \leq i \leq 33)
\right\rbrace.
\end{eqnarray*}
\end{small}
Since  $m(He)$, $M(He)$, $\mathcal{C}_{1} = (\mathcal{X}_{1},
\mathcal{K}_{1})$, $\mathcal{C}_{2} =
(\mathcal{X}_{2},\mathcal{K}_{2})$ and  $\mathcal{C}_{3} =
\mathcal{C}_{1} \vee \mathcal{C}_{2}$ are five supercharacter
theories for $He$, $s(He) \geq 5$.

\item {\textit{The Harada-Norton group $HN$}}. This group has exactly 54 conjugacy classes and irreducible characters. Suppose $Irr(HN) = \{  \chi_i\}_{1 \leq i \leq 54}$ and
$Con(HN) = \{  x_i^{HN}\}_{1 \leq i \leq 54}$. We also define:
\begin{small}
\begin{eqnarray*}
\mathcal{K}_{1} &=& \left\lbrace \{1\}, \{x_{i}^{HN}\} (2 \leq i \leq 38), \{x_{39}^{HN},x_{40}^{HN}\}, \{x_{i}^{HN}\} (41 \leq i \leq 54) \right\rbrace, \\
\mathcal{X}_{1} &=& \left\lbrace \{1\}, \{\chi_{i}\} (2 \leq i \leq 50), \{\chi_{51}, \chi_{52}\}, \{\chi_{i}\} (53 \leq i \leq 54) \right\rbrace, \\
\mathcal{K}_{2} &=& \left\lbrace \{1\}, \{x_{i}^{HN}\} (2 \leq i \leq 52), \{x_{53}^{HN},x_{54}^{HN}\}\right\rbrace, \\
\mathcal{X}_{2} &=& \left\lbrace \{1\}, \{\chi_{i}\} (2 \leq i
\leq 34), \{\chi_{35}, \chi_{36}\}, \{\chi_{i}\} (37 \leq i \leq
54) \right\rbrace.
\end{eqnarray*}
\end{small}
By Lemma \ref{11} and  \cite[Proposition 2.16]{10},
$\mathcal{C}_{1} = (\mathcal{X}_{1}, \mathcal{K}_{1})$,
$\mathcal{C}_{2} = (\mathcal{X}_{2},\mathcal{K}_{2})$ and
$\mathcal{C}_{3} = \mathcal{C}_{1} \vee \mathcal{C}_{2}$ are
supercharacter theories of $HN$ and apart from supercharacter
theories $m(HN)$ and $M(HN)$, it concludes that $m(HN) \geq 5$.

\item {\textit{The Thompson group $Th$}}. The Thompson group $Th$ has exactly 48 conjugacy classes and irreducible characters. We assume that
$Irr(Th) = \{  \chi_i\}_{1 \leq i \leq 48}$ and $Con(Th) = \{
x_i^{Th}\}_{1 \leq i \leq 48}$ and define:
{\small\begin{eqnarray*}
\mathcal{K}_{1} &=& \left\lbrace \{1\}, \{x_{i}^{Th}\} (2 \leq i \leq 24), \{x_{25}^{Th}, x_{26}^{Th}\}, \{x_{i}^{Th}\} (27 \leq i \leq 39), \{x_{40}^{Th}, x_{41}^{Th}\}, \{x_{i}^{Th}\} (42 \leq i \leq 48) \right\rbrace, \\
\mathcal{X}_{1} &=& \left\lbrace \{1\}, \{\chi_{i}\} (2 \leq i
\leq 8), \{\chi_{9}, \chi_{10}\}, \{\chi_{i}\} (11 \leq i \leq
34), \{\chi_{35}, \chi_{36}\}, \{\chi_{i}\} (37
\leq i \leq 48) \right\rbrace, \\
\mathcal{K}_{2} &=& \left\lbrace \{1\}, \{x_{i}^{Th}\} (2 \leq i \leq 36), \{x_{37}^{Th}, x_{38}^{Th}\}, \{x_{i}^{Th}\} (39 \leq i \leq 48) \right\rbrace ,\\
\mathcal{X}_{2} &=& \left\lbrace \{1\}, \{\chi_{i}\} (2 \leq i
\leq 21), \{\chi_{22}, \chi_{23}\}, \{\chi_{i}\} (24 \leq i \leq
48) \right\rbrace.
\end{eqnarray*}}
Again apply Lemma \ref{11} and  \cite[Proposition 2.16]{10} to
deduce that  $\mathcal{C}_{1} = (\mathcal{X}_{1},
\mathcal{K}_{1})$,  $\mathcal{C}_{2} =
(\mathcal{X}_{2},\mathcal{K}_{2})$ and finally $\mathcal{C}_{3} =
\mathcal{C}_{1} \vee \mathcal{C}_{2}$ are supercharacter theories
for $Th$. Thus $s(Th) \geq 5$.

\item {\textit{The Fischer group $Fi_{22}$}}. The Fischer group $Fi_{22}$
has exactly $65$ conjugacy classes and irreducible characters. Set
$Irr(Fi_{22}) = \{ \chi_i\}_{1 \leq i \leq 65}$ and $Con(Fi_{22})
= \{ x_i^{Fi_{22}}\}_{1 \leq i \leq 65}$. We also define:
\begin{scriptsize}
\begin{eqnarray*}
\mathcal{K}_{1} &=& \left\lbrace \{1\}, \{x_{i}^{Fi_{22}}\} (2
\leq i \leq 35), \{x_{36}^{Fi_{22}}, x_{37}^{Fi_{22}}\},
\{x_{i}^{Fi_{22}}\} (38 \leq i \leq 60),
\{x_{61}^{Fi_{22}}, x_{62}^{Fi_{22}}\}, \{x_{i}^{Fi_{22}}\} (63 \leq i \leq 65) \right\rbrace, \\
\mathcal{X}_{1} &=& \left\lbrace \{1\}, \{\chi_{i}\} (2 \leq i \leq 39), \{\chi_{40}, \chi_{41}\}, \{\chi_{i}\} (42 \leq i \leq 50), \{\chi_{51}, \chi_{52}\}, \{\chi_{i}\} (53 \leq i \leq 65) \right\rbrace, \\
\mathcal{K}_{2} &=& \left\lbrace \{1\}, \{x_{i}^{Fi_{22}}\} (2 \leq i \leq 54), \{x_{55}^{Fi_{22}}, x_{56}^{Fi_{22}}\}, \{x_{i}^{Fi_{22}}\} (57 \leq i \leq 65)\right\rbrace, \\
\mathcal{X}_{2} &=& \left\lbrace \{1\}, \{\chi_{i}\} (2 \leq i
\leq 42), \{\chi_{43}, \chi_{44}\}, \{\chi_{i}\} (45 \leq i \leq
65) \right\rbrace.
\end{eqnarray*}
\end{scriptsize}
Since $\mathcal{C}_{1} = (\mathcal{X}_{1}, \mathcal{K}_{1})$,
$\mathcal{C}_{2} = (\mathcal{X}_{2},\mathcal{K}_{2})$ and finally
$\mathcal{C}_{3} = \mathcal{C}_{1} \vee \mathcal{C}_{2}$, are
three supercharacter theories of $Fi_{22}$ different from
$m(Fi_{22})$ and $M(Fi_{22})$, $s(Fi_{22}) \geq 5$.

\item {\textit{The Fischer group $Fi_{23}$}}. This group has exactly $98$ conjugacy classes and irreducible characters.
Set $Irr(Fi_{23}) = \{  \chi_i\}_{1 \leq i \leq 98}$ and
$Con(Fi_{23}) = \{  x_i^{Fi_{23}}\}_{1 \leq i \leq 98}$. Define:
\begin{footnotesize}
\begin{eqnarray*}
\mathcal{K}_{1} &=& \left\lbrace \{1\}, \{x_{i}^{Fi_{23}}\} (2 \leq i \leq 79), \{x_{80}^{Fi_{23}}, x_{81}^{Fi_{23}}\}, \{x_{i}^{Fi_{23}}\} (82 \leq i \leq 98)\right\rbrace, \\
\mathcal{X}_{1} &=& \left\lbrace \{1\}, \{\chi_{i}\} (2 \leq i \leq 16), \{\chi_{17}, \chi_{18}\}, \{\chi_{i}\} (19 \leq i \leq 98) \right\rbrace, \\
\mathcal{K}_{2} &=& \left\lbrace \{1\}, \{x_{i}^{Fi_{23}}\} (2 \leq i \leq 62), \{x_{63}^{Fi_{23}}, x_{64}^{Fi_{23}}\}, \{x_{i}^{Fi_{23}}\} (65 \leq i \leq 79), \{x_{80}^{Fi_{23}}, x_{81}^{Fi_{23}}\}, \{x_{i}^{Fi_{23}}\} (82 \leq i \leq 98)\right\rbrace, \\
\mathcal{X}_{2} &=& \left\lbrace \{1\}, \{\chi_{i}\} (2 \leq i
\leq 14), \{\chi_{15}, \chi_{16}\}, \{\chi_{17}, \chi_{18}\},
\{\chi_{i}\} (19 \leq i \leq 98) \right\rbrace,\\
\mathcal{C}_{1} &=& (\mathcal{X}_{1}, \mathcal{K}_{1}),\ \
\mathcal{C}_{2} = (\mathcal{X}_{2},\mathcal{K}_{2}),\ \
\mathcal{C}_{3} = \mathcal{C}_{1} \vee \mathcal{C}_{2}.
\end{eqnarray*}
\end{footnotesize}
Since $\{ m(Fi_{23}), M(Fi_{23}), \mathcal{C}_{1},
\mathcal{C}_{2}, \mathcal{C}_{3}\} \subseteq Sup(Fi_{23})$,
$s(Fi_{23}) \geq 5$.

\item {\textit{The Fischer group $Fi_{24}^{\prime}$}}. The largest Fischer group has exactly $108$ conjugacy classes and irreducible characters. Set
$Irr(Fi_{24}^{\prime}) = \{  \chi_i\}_{1 \leq i \leq 108}$ and
$Con(Fi_{24}^{\prime}) = \{  x_i^{Fi_{24}^{\prime}}\}_{1 \leq i
\leq 108}$. Define:
\begin{small}
\begin{eqnarray*}
\mathcal{K}_{1} &=& \left\lbrace \{1\},
\{x_{i}^{Fi_{24}^{\prime}}\} (2 \leq i \leq 105),
\{x_{106}^{Fi_{24}^{\prime}}, x_{107}^{Fi_{24}^{\prime}}\},
\{x_{108}^{Fi_{24}^{\prime}}\}\right\rbrace, \\
\mathcal{X}_{1} &=& \left\lbrace \{1\}, \{\chi_{i}\} (2 \leq i \leq 98), \{\chi_{99}, \chi_{100}\}, \{\chi_{i}\} (101 \leq i \leq 108) \right\rbrace, \\
\mathcal{K}_{2} &=& \left\lbrace \{1\},
\{x_{i}^{Fi_{24}^{\prime}}\} (2 \leq i \leq 80),
\{x_{81}^{Fi_{24}^{\prime}}, x_{82}^{Fi_{24}^{\prime}}\},
\{x_{i}^{Fi_{24}^{\prime}}\} (83 \leq i \leq 108)\right\rbrace, \\
\mathcal{X}_{2} &=& \left\lbrace \{1\}, \{\chi_{i}\} (2 \leq i
\leq 100), \{\chi_{101}, \chi_{102}\}, \{\chi_{i}\} (103 \leq i
\leq 108) \right\rbrace.
\end{eqnarray*}
\end{small}
Since  $m(Fi_{24}^{\prime})$, $M(Fi_{24}^{\prime})$,
$\mathcal{C}_{1} = (\mathcal{X}_{1}, \mathcal{K}_{1})$,
$\mathcal{C}_{2} = (\mathcal{X}_{2},\mathcal{K}_{2})$ and
$\mathcal{C}_{3} = \mathcal{C}_{1} \vee \mathcal{C}_{2}$ are
supercharacter theories of $Fi_{24}$, $s(Fi_{24}) \geq 5$.

\item {\textit{The Baby Monster group $B$}}. This group has exactly $184$ conjugacy classes and irreducible characters.
Suppose $Irr(B) = \{  \chi_i\}_{1 \leq i \leq 184}$, $Con(B) = \{
x_i^{B}\}_{1 \leq i \leq 184}$ and define:
\begin{small}
\begin{eqnarray*}
\mathcal{K}_{1} &=& \left\lbrace \{1\}, \{x_{i}^{B}\} (2 \leq i \leq 177), \{x_{178}^{B}, x_{179}^{B}\}, \{x_{i}^{B}\} (180 \leq i \leq 184)\right\rbrace, \\
\mathcal{X}_{1} &=& \left\lbrace \{1\}, \{\chi_{i}\} (2 \leq i
\leq 177), \{\chi_{178}, \chi_{179}\}, \{\chi_{i}\} (180 \leq i
\leq 184) \right\rbrace,\\
\mathcal{K}_{2} &=& \left\lbrace \{1\}, \{x_{i}^{B}\} (2 \leq i \leq 171), \{x_{172}^{B}, x_{173}^{B}\}, \{x_{i}^{B}\} (174 \leq i \leq 184)\right\rbrace, \\
\mathcal{X}_{2} &=& \left\lbrace \{1\}, \{\chi_{i}\} (2 \leq i
\leq 20), \{\chi_{21}, \chi_{22}\}, \{\chi_{i}\} (23 \leq i \leq
184) \right\rbrace.
\end{eqnarray*}
\end{small}
Since $m(B)$, $M(B)$, $\mathcal{C}_{1} = (\mathcal{X}_{1},
\mathcal{K}_{1})$,  $\mathcal{C}_{2} =
(\mathcal{X}_{2},\mathcal{K}_{2})$ and  $\mathcal{C}_{3} =
\mathcal{C}_{1} \vee \mathcal{C}_{2}$ are supercharacter theories
of $B$, $s(B) \geq 5$, as desired.

\item {\textit{The Monster group $M$}}. The largest sporadic group
$M$ has exactly $198$ conjugacy classes and irreducible characters.
Set $Irr(M) = \{ \chi_i\}_{1 \leq i \leq 194}$ and $Con(M) = \{
x_i^{M}\}_{1 \leq i \leq 194}$.
\begin{small}
\begin{eqnarray*}
\mathcal{K}_{1} &=& \left\lbrace \{1\}, \{x_{i}^{M}\} (2 \leq i \leq 192), \{x_{193}^{M}, x_{194}^{M}\}\right\rbrace, \\
\mathcal{X}_{1} &=& \left\lbrace \{1\}, \{\chi_{i}\} (2 \leq i
\leq 46), \{\chi_{47}, \chi_{48}\}, \{\chi_{i}\} (49 \leq i \leq
194) \right\rbrace,\\
\mathcal{K}_{2} &=& \left\lbrace \{1\}, \{x_{i}^{M}\} (2 \leq i \leq 188), \{x_{189}^{M}, x_{190}^{M}\}, \{x_{i}^{M}\} (191 \leq i \leq 194)\right\rbrace, \\
\mathcal{X}_{2} &=& \left\lbrace \{1\}, \{\chi_{i}\} (2 \leq i \leq 123), \{\chi_{124}, \chi_{125}\}, \{\chi_{i}\} (126 \leq i \leq 194) \right\rbrace.
\end{eqnarray*}
\end{small}
We can see that the group $M$ has at least $5$ supercharacter
theories as $m(M)$, $M(M)$, $\mathcal{C}_{1} = (\mathcal{X}_{1},
\mathcal{K}_{1})$,  $\mathcal{C}_{2} =
(\mathcal{X}_{2},\mathcal{K}_{2})$ and  $\mathcal{C}_{3} =
\mathcal{C}_{1} \vee \mathcal{C}_{2}$. This shows that $s(M) \geq
5$.
\end{itemize}
This completes the proof.
\end{proof}

\begin{thm}
The Pariahs have at least five supercharacter theories.
\end{thm}

\begin{proof}
The Pariahs are six sporadic groups $J_1$, $O^\prime{N}$, $J_3$,
$Ru$, $J_4$ and $Ly$. Our main proof will consider six separate
cases as follows:
\begin{itemize}

\item \textit{Janko group $J_{1}$}. The first Janko group $J_1$ has exactly $15$ conjugacy classes and irreducible characters. Set $Irr(J_{1}) = \{  \chi_i\}_{1 \leq i \leq 15}$ and
$Con(J_{1}) = \{  x_i^{J_{1}}\}_{1 \leq i \leq 15}$.
\begin{footnotesize}
\begin{eqnarray*}
\mathcal{K}_{1} &=& \left\lbrace \{1\}, \{x_{i}^{J_{1}}\} (2 \leq
i \leq 3), \{x_{4}^{J_{1}}, x_{5}^{J_{1}}\}, \{x_{i}^{J_{1}}\} (6
\leq i \leq 7), \{x_{8}^{J_{1}}, x_{9}^{J_{1}}\},
\{x_{10}^{J_{1}}\},
\{x_{11}^{J_{1}}, x_{12}^{J_{1}}\}, \{x_{i}^{J_{1}}\} (13 \leq i \leq 15)\right\rbrace, \\
\mathcal{X}_{1} &=& \left\lbrace \{1\}, \{\chi_{2}, \chi_{3}\},
\{\chi_{i}\} (4 \leq i \leq 6), \{\chi_{7}, \chi_{8}\},
\{\chi_{i}\} (9 \leq i \leq 12), \{\chi_{13}, \chi_{14}\},
\{\chi_{15}\} \right\rbrace,\\
\mathcal{K}_{2} &=& \left\lbrace \{1\}, \{x_{i}^{J_{1}}\} (2 \leq i \leq 12), \{x_{13}^{J_{1}}, x_{14}^{J_{1}}, x_{15}^{J_{1}}\} \right\rbrace, \\
\mathcal{X}_{2} &=& \left\lbrace \{1\}, \{\chi_{i}\} (2 \leq i
\leq 8), \{\chi_{9}, \chi_{10}, \chi_{11}\}, \{\chi_{i}\} (12
\leq i \leq 15) \right\rbrace.
\end{eqnarray*}
\end{footnotesize}
Since  $m(J_{1})$, $M(J_{1})$, $\mathcal{C}_{1} =
(\mathcal{X}_{1}, \mathcal{K}_{1})$,  $\mathcal{C}_{2} =
(\mathcal{X}_{2},\mathcal{K}_{2})$ and $\mathcal{C}_{3} =
\mathcal{C}_{1} \vee \mathcal{C}_{2}$ are supercharacter theories
of $J_1$, $s(J_1) \geq 5$.

\item {\textit{The O$^{\prime}$Nan group $O^{\prime}N$}}. The O$^{\prime}$Nan group $O^{\prime}N$ has exactly $30$ conjugacy classes and irreducible characters.
Set $Irr(O^{\prime}N) = \{  \chi_i\}_{1 \leq i \leq 30}$ and
$Con(O^{\prime}N) = \{  x_i^{O^{\prime}N}\}_{1 \leq i \leq 30}$.
\begin{small}
\begin{eqnarray*}
\mathcal{K}_{1} &=& \left\lbrace \{1\}, \{x_{i}^{O^{\prime}N}\}
(2 \leq i \leq 21), \{x_{22}^{O^{\prime}N}, x_{23}^{O^{\prime}N},
x_{24}^{O^{\prime}N}\},
\{x_{i}^{O^{\prime}N}\} (25 \leq i \leq 30) \right\rbrace, \\
\mathcal{X}_{1} &=& \left\lbrace \{1\}, \{\chi_{i}\} (2 \leq i
\leq 25), \{\chi_{26}, \chi_{27}, \chi_{28}\}, \{\chi_{i}\} (29
\leq i \leq 30) \right\rbrace, \\
\mathcal{K}_{2} &=& \left\lbrace \{1\}, \{x_{i}^{O^{\prime}N}\} (2 \leq i \leq 26), \{x_{27}^{O^{\prime}N}, x_{28}^{O^{\prime}N}\}, \{x_{i}^{O^{\prime}N}\}
(29 \leq i \leq 30) \right\rbrace, \\
\mathcal{X}_{2} &=& \left\lbrace \{1\}, \{\chi_{i}\} (2 \leq i
\leq 28), \{\chi_{29}, \chi_{30}\}\right\rbrace.
\end{eqnarray*}
\end{small}
Apply Lemma \ref{11} and  \cite[Proposition 2.16]{10} to deduce
that $\mathcal{C}_{1} = (\mathcal{X}_{1}, \mathcal{K}_{1})$,
$\mathcal{C}_{2} = (\mathcal{X}_{2},\mathcal{K}_{2})$ and
$\mathcal{C}_{3} = \mathcal{C}_{1} \vee \mathcal{C}_{2}$ are
supercharacter theories of $O^{\prime}N$. Hence $s(O^{\prime}N)
\geq 5$, as desired.

\item {\textit{The Janko group $J_{3}$}}. We assume that
$Irr(J_{3}) = \{  \chi_i\}_{1 \leq i \leq 21}$ and $Con(J_{3}) =
\{  x_i^{J_{3}}\}_{1 \leq i \leq 21}$. Define:
\begin{small}
\begin{eqnarray*}
\mathcal{K}_{1} &=& \left\lbrace \{1\}, \{x_{i}^{J_{3}}\} (2 \leq i \leq 9), \{x_{10}^{J_{3}}, x_{11}^{J_{3}}, x_{12}^{J_{3}}\}, \{x_{i}^{J_{3}}\} (13 \leq i \leq 21) \right\rbrace, \\
\mathcal{X}_{1} &=& \left\lbrace \{1\}, \{\chi_{i}\} (2 \leq i
\leq 13), \{\chi_{14}, \chi_{15}, \chi_{16}\}, \{\chi_{i}\} (17
\leq i \leq 21) \right\rbrace, \\
\mathcal{K}_{2} &=& \left\lbrace \{1\}, \{x_{i}^{J_{3}}\} (2 \leq i \leq 19), \{x_{20}^{J_{3}}, x_{21}^{J_{3}}\} \right\rbrace, \\
\mathcal{X}_{2} &=& \left\lbrace \{1\}, \{\chi_{2}, \chi_{3}\},
\{\chi_{i}\} (4 \leq i \leq 21) \right\rbrace.
\end{eqnarray*}
\end{small}
Since  $m(J_{3})$, $M(J_{3})$, $\mathcal{C}_{1} =
(\mathcal{X}_{1}, \mathcal{K}_{1})$,  $\mathcal{C}_{2} =
(\mathcal{X}_{2},\mathcal{K}_{2})$ and $\mathcal{C}_{3} =
\mathcal{C}_{1} \vee \mathcal{C}_{2}$ are supercharacter theories
of $J_3$, $s(J_3) \geq 5$.

\item {\textit{The Rudvalis group $Ru$}}. Set $Irr(Ru) = \{  \chi_i\}_{1 \leq i \leq 36}$ and $Con(Ru) = \{  x_i^{Ru}\}_{1 \leq i \leq 36}$ and
define:
\begin{small}
\begin{eqnarray*}
\mathcal{K}_{1} &=& \left\lbrace \{1\}, \{x_{i}^{Ru}\} (2 \leq i \leq 34), \{x_{35}^{Ru}, x_{36}^{Ru}\}\right\rbrace, \\
\mathcal{X}_{1} &=& \left\lbrace \{1\}, \{\chi_{i}\} (2 \leq i
\leq 33), \{\chi_{34}, \chi_{35}\}, \{\chi_{36}\} \right\rbrace,\\
\mathcal{K}_{2} &=& \left\lbrace \{1\}, \{x_{i}^{Ru}\} (2 \leq i \leq 31), \{x_{32}^{Ru}, x_{33}^{Ru}, x_{34}^{Ru}\}, \{x_{i}^{Ru}\} (35 \leq i \leq 36) \right\rbrace, \\
\mathcal{X}_{2} &=& \left\lbrace \{1\}, \{\chi_{i}\} (2 \leq i
\leq 16), \{\chi_{17}, \chi_{18}, \chi_{19}\}, \{\chi_{i}\} (20
\leq i \leq 36) \right\rbrace.
\end{eqnarray*}
\end{small}
Apply again Lemma \ref{11} and  \cite[Proposition 2.16]{10} to
deduce that the Rudvalis group $Ru$ has at least five
supercharacter theories as $m(Ru)$, $M(Ru)$, $\mathcal{C}_{1} =
(\mathcal{X}_{1}, \mathcal{K}_{1})$, $\mathcal{C}_{2} =
(\mathcal{X}_{2},\mathcal{K}_{2})$ and  $\mathcal{C}_{3} =
\mathcal{C}_{1} \vee \mathcal{C}_{2}$. Hence $s(Ru) \geq 5$, as
required.

\item {\textit{The Janko group $J_{4}$}}. This group has exactly $62$ conjugacy classes and irreducible characters. Define
$Irr(J_{4}) = \{  \chi_i\}_{1 \leq i \leq 62}$, $Con(J_{4}) = \{
x_i^{J_{4}}\}_{1 \leq i \leq 62}$ and define:
\begin{footnotesize}
\begin{eqnarray*}
\mathcal{K}_{1} &=& \left\lbrace \{1\}, \{x_{i}^{J_{4}}\} (2 \leq i \leq 49), \{x_{50}^{J_{4}}, x_{51}^{J_{4}}, x_{52}^{J_{4}}\}, \{x_{i}^{J_{4}}\} (53 \leq i \leq 62)\right\rbrace, \\
\mathcal{X}_{1} &=& \left\lbrace \{1\}, \{\chi_{i}\} (2 \leq i
\leq 52), \{\chi_{53}, \chi_{54}, \chi_{55}\}, \{\chi_{i}\} ( 56
\leq i \leq 62) \right\rbrace, \\
\mathcal{K}_{2} &=&
\left\lbrace \{1\}, \{x_{i}^{J_{4}}\} (2 \leq i \leq 42),
\{x_{43}^{J_{4}}, x_{44}^{J_{4}}, x_{45}^{J_{4}}\},
\{x_{i}^{J_{4}}\} (46 \leq i \leq 49),
 \{x_{50}^{J_{4}}, x_{51}^{J_{4}}, x_{52}^{J_{4}}\},  \{x_{i}^{J_{4}}\} (53 \leq i \leq 62) \right\rbrace, \\
\mathcal{X}_{2} &=& \left\lbrace \{1\}, \{\chi_{i}\} (2 \leq i
\leq 52), \{\chi_{53}, \chi_{54}, \chi_{55}\}, \{\chi_{56},
\chi_{57}, \chi_{58}\}, \{\chi_{i}\} (59 \leq i \leq 62)
\right\rbrace.
\end{eqnarray*}
\end{footnotesize}
The Janko group $J_{4}$ has at least five supercharacter theories
as $m(J_{4})$, $M(J_{4})$, $\mathcal{C}_{1} = (\mathcal{X}_{1},
\mathcal{K}_{1})$,  $\mathcal{C}_{2} =
(\mathcal{X}_{2},\mathcal{K}_{2})$ and $\mathcal{C}_{3} =
\mathcal{C}_{1} \vee \mathcal{C}_{2}$. Therefore, $s(J_4) \geq 5$.

\item \textit{The Lyons group $Ly$}. The Lyons group $Ly$ has exactly $53$ irreducible characters and conjugacy classes. Set $Irr(Ly) = \{  \chi_i\}_{1 \leq i \leq 53}$
and $Con(Ly) = \{  x_i^{Ly}\}_{1 \leq i \leq 53}$ and define:
\begin{footnotesize}
\begin{eqnarray*}
\mathcal{K}_{1} &=& \left\lbrace \{1\}, \{x_{i}^{Ly}\} (2 \leq i
\leq 37), \{x_{38}^{Ly}, x_{39}^{Ly}, x_{40}^{Ly}, x_{41}^{Ly},
x_{42}^{Ly}\},
\{x_{i}^{Ly}\} (43 \leq i \leq 53)\right\rbrace, \\
\mathcal{X}_{1} &=& \left\lbrace \{1\}, \{\chi_{i}\} (2 \leq i
\leq 38), \{\chi_{39}, \chi_{40}, \chi_{41}, \chi_{42},
\chi_{43}\}, \{\chi_{i}\} ( 44 \leq i \leq 53) \right\rbrace, \\
\mathcal{K}_{2} &=& \left\lbrace \{1\}, \{x_{i}^{Ly}\} (2 \leq i \leq 37), \{x_{38}^{Ly}, x_{39}^{Ly}, x_{40}^{Ly}, x_{41}^{Ly}, x_{42}^{Ly}\}, \{x_{i}^{Ly}\} (43 \leq i \leq 50), \{x_{51}^{Ly}, x_{52}^{Ly}, x_{53}^{Ly}\}\right\rbrace, \\
\mathcal{X}_{2} &=& \left\lbrace \{1\}, \{\chi_{i}\} (2 \leq i
\leq 25), \{\chi_{26}, \chi_{27}, \chi_{28}\}, \{\chi_{i}\} (29
\leq i \leq 38),
\{\chi_{39}, \chi_{40}, \chi_{41}, \chi_{42}, \chi_{43}\}, \{\chi_{i}\} (44 \leq i \leq 53) \right\rbrace.
\end{eqnarray*}
\end{footnotesize}
Since the group $Ly$ has at least five supercharacter theories as
$m(Ly)$, $M(Ly)$, $\mathcal{C}_{1} = (\mathcal{X}_{1},
\mathcal{K}_{1})$,  $\mathcal{C}_{2} =
(\mathcal{X}_{2},\mathcal{K}_{2})$ and $\mathcal{C}_{3} =
\mathcal{C}_{1} \vee \mathcal{C}_{2}$, $s(Ly) \geq 5$.
\end{itemize}
This completes the proof.
\end{proof}

\begin{lem}
The  Janko group $J_{2}$ has at least three supercharacter
theories.
\end{lem}

\begin{proof}
A simple investigation shows that the supercharacter theories of
the Janko group $J_2$ are $m(J_{2})$, $M(J_{2})$, $\mathcal{C} =
(\mathcal{X}, \mathcal{K})$ such that
\begin{footnotesize}
\begin{eqnarray*}
\mathcal{K} &=& \left\lbrace \{1\}, \{x_{i}^{J_{2}}\} (2 \leq i
\leq 6), \{x_{7}^{J_{2}},x_{8}^{J_{2}}\},
\{x_{9}^{J_{2}},x_{10}^{J_{2}}\}, \{x_{i}^{J_{2}}\}
 (11 \leq i \leq 14), \{x_{15}^{J_{2}},x_{16}^{J_{2}}\}, \{x_{17}^{J_{2}},x_{18}^{J_{2}}\}, \{x_{19}^{J_{2}}\}, \{x_{20}^{J_{2}},x_{21}^{J_{2}}\} \right\rbrace, \\
\mathcal{X} &=& \left\lbrace \{1\}, \{\chi_{2}, \chi_{3}\},
\{\chi_{4}, \chi_{5}\}, \{\chi_{i}\} (6 \leq i \leq 7),
\{\chi_{8}, \chi_{9}\},\{\chi_{i}\}
 (10 \leq i \leq 13), \{\chi_{14}, \chi_{15}\}, \{\chi_{16}, \chi_{17}\}, \{\chi_{i}\} (18 \leq i \leq 20)\right\rbrace,
\end{eqnarray*}
\end{footnotesize}
proving the lemma.
\end{proof}

We run a GAP program to be sure that $s(J_2) = 3$, but our
program after some days stopped, because it needs a huge amount
of RAM.

\begin{conj}
$s(J_2) = 3$.
\end{conj}

\section{Supercharacter Theory Construction for Alternating and Suzuki  Groups}

The intention of this section is to move a step towards a
classification of  the finite simple groups with exactly three
and four supercharacter theories. We start with the cyclic group
of order $p$.

\begin{lem}\label{12}
Suppose  $p$ is prime. Then,
\begin{enumerate}
\item $s(Z_{p}) = 3$ if and only if $p = 5$,

\item $s(Z_{p}) = 4$ if and only if $p$ is a Sophie Germain prime.
\end{enumerate}
\end{lem}

\begin{proof}
It is clear that $p$ is an odd prime. By  \cite[Theorem 6.32 and
Table 1]{9}, $s(Z_p) = d(p-1)$.

\begin{enumerate}
\item \textit{$s(Z_{p}) = 3$.} In this case,
$s(Z_{p}) = d(p-1)$ and so $p \geq 5$. Since $p-1$ is an even
integer, the case of $p > 5$ cannot be occurred and so $p = 5$, as
desired.

\item \textit{$s(Z_{p}) = 4$.} Since $d(p-1) = 4$, $p-1 = q^3$, $q$ is prime, or $p-1 = 2r$,
where $r$ is prime. If $p-1 = q^3$ then $q = 2$ and $p = 9$, a
contradiction. So, $p-1 = 2r$, where $r$ is prime. This shows
that $p$ is a Sophie Germain prime.
\end{enumerate}
This completes the proof.
\end{proof}

The following well-known results are crucial in the
classification of alternating simple groups with exactly three or
four supercharacter theories.

\begin{thm}\label{4.1} The following are hold:
\begin{enumerate}
\item {\rm(Berggren \cite{5})} Every irreducible characters of the
alternating group $A_n$ are real valued if and only if $n \in \{
1, 2, 5, 6, 10, 14\}$.

\item {\rm(Grove \cite[Proposition 8.2.1]{85})} If $K$ is a conjugacy
class in $S_{n}$, $K \subset A_{n}$, and $\sigma \in K$, then $K$
is a conjugacy class in $A_{n}$ if and only if some odd elements
of $S_{n}$ commutes with $\sigma$; if that is not the case, then
the conjugacy class $K$ splits as the union of two
$A_{n}-$classes, each of size $|K|/2$. If $\lambda$ is the
(partition) type of $\sigma$ then $K$ splits if and only if the
parts of $\lambda$ are all odd and all different from each other.

\item Suppose $x \in A_n$  is a product of $r$ pair-wise disjoint
cycles including all fixed points as singleton cycles. Then $x^{A_{n}}$
is non-real if and only if $\sum_{j=1}^{m}\frac{r_j-1}{2}$  is
odd.
\end{enumerate}
\end{thm}

\begin{thm}
The simple alternating group $A_{n}$ has exactly three
supercharacter theories if and only if $n = 5$ or $7$. There is
no simple alternating groups with exactly four supercharacter
theories.
\end{thm}

\begin{proof}
Suppose $n \geq 5$. It is easy to see that the alternating groups
$A_5$ and $A_7$ have exactly three supercharacter theories. Our
main proof will consider three separate cases as follows:

\begin{enumerate}
\item \textit{All character values of $A_n$ are real}.  Since $n \geq 5$, Theorem
\ref{4.1}(1) implies that $n = 5, 6, 10$ or $14.$ In what follows
three non-trivial supercharacter theories for the alternating
groups $A_6, A_{10}$ and $A_{14}$ are presented.
\begin{enumerate}
\item By a GAP program, one can see that the group $A_{5}$ has exactly $3$ supercharacter theories
$m(A_{5})$, $M(A_{5})$ and $\mathcal{C} = (\mathcal{X},
\mathcal{K})$ such that
\begin{small}
\begin{eqnarray*}
\mathcal{K} &=& \left\lbrace \{1\}, \{x_{2}^{A_{5}}\}, \{x_{3}^{A_{5}}\}, \{x_{4}^{A_{5}}, x_{5}^{A_{5}}\}\right\rbrace, \\
\mathcal{X} &=& \left\lbrace \{1\}, \{\chi_{2}, \chi_{3}\},
\{\chi_{i}\} (4 \leq i \leq 5) \right\rbrace.
\end{eqnarray*}
\end{small}
Here, $Irr(A_{5})$ = $\{  \chi_1, \chi_2, \chi_3, \chi_4,
\chi_{5}\}$ and $Con(A_{5})$ = $\{  x_1^{A_{5}}, x_2^{A_{5}},
x_3^{A_{5}}, x_4^{A_{5}}, x_{5}^{A_{5}}\}$.

\item Suppose $Irr(A_{6})$ = $\{  \chi_1, \chi_2, \ldots,
\chi_{7}\}$ and $Con(A_{6})$ = $\{  x_1^{A_{6}}, x_2^{A_{6}},
\ldots, x_{7}^{A_{6}}\}$. Define $\mathcal{C}_{1} =
(\mathcal{X}_{1}, \mathcal{K}_{1})$ and $\mathcal{C}_{2} =
(\mathcal{X}_{2}, \mathcal{K}_{2})$ as follows:
\begin{small}
\begin{eqnarray*}
\mathcal{K}_{1} &=& \left\lbrace \{1\}, \{x_{2}^{A_{6}}\}, \{x_{3}^{A_{6}}, x_{4}^{A_{6}}\}, \{x_{5}^{A_{6}}\}, \{x_{6}^{A_{6}}, x_{7}^{A_{6}}\}\right\rbrace, \\
\mathcal{X}_{1} &=& \left\lbrace \{1\}, \{\chi_{2}, \chi_{3}\},
\{\chi_{4}, \chi_{5}\}, \{\chi_{6}\}, \{\chi_{7}\}\right\rbrace,\\
\mathcal{K}_{2} &=& \left\lbrace \{1\}, \{x_{2}^{A_{6}}, x_{5}^{A_{6}}\}, \{x_{3}^{A_{6}}, x_{4}^{A_{6}}\}, \{x_{6}^{A_{6}}\}, \{x_{7}^{A_{6}}\}\right\rbrace, \\
\mathcal{X}_{2} &=& \left\lbrace \{1\}, \{\chi_{2}, \chi_{3},
\chi_{7}\}, \{\chi_{4}\}, \{\chi_{5}\}, \{\chi_{6}\}\right\rbrace,
\end{eqnarray*}
\end{small}
If $\mathcal{C}_{3} = \mathcal{C}_{1} \vee \mathcal{C}_{2}$ then
$\mathcal{C}_{3} $ is a  supercharacter theory for $A_6$
different from $\mathcal{C}_1$, $\mathcal{C}_2$, $m(A_{6})$ and
$M(A_{6})$. Therefore, $A_{6}$ has at least five supercharacter
theories.

\item Suppose $Irr(A_{10})$ = $\{  \chi_1, \chi_2, \ldots, \chi_{24}\}$,  $Con(A_{10}) = \{  x_1^{A_{10}}, x_2^{A_{10}}, \ldots,
x_{24}^{A_{10}}\}$, $\mathcal{C}_{1} =
(\mathcal{X}_{1},\mathcal{K}_{1})$, $\mathcal{C}_{2} =
(\mathcal{X}_{2},\mathcal{K}_{2})$ and
 $\mathcal{C}_{3} = \mathcal{C}_{1} \vee \mathcal{C}_{2}$, where
\begin{small}
\begin{eqnarray*}
\mathcal{K}_{1} &=& \left\lbrace \{1\}, \{x_{i}^{A_{10}}\} (2 \leq i \leq 19), \{x_{20}^{A_{10}}, x_{21}^{A_{10}}\}, \{x_{i}^{A_{10}}\} (22 \leq i \leq 24)\right\rbrace, \\
\mathcal{X}_{1} &=& \left\lbrace \{1\}, \{\chi_{i}\} (2 \leq i
\leq 19), \{\chi_{20}, \chi_{21}\}, \{\chi_{i}\} (22 \leq i \leq
24) \right\rbrace,\\
\mathcal{K}_{2} &=& \left\lbrace \{1\}, \{x_{i}^{A_{10}}\} (2 \leq i \leq 22), \{x_{23}^{A_{10}}, x_{24}^{A_{10}}\}\right\rbrace, \\
\mathcal{X}_{2} &=& \left\lbrace \{1\}, \{\chi_{i}\} (2 \leq i
\leq 11), \{\chi_{12}, \chi_{13}\}, \{\chi_{i}\} (14 \leq i \leq
24) \right\rbrace.
\end{eqnarray*}
\end{small}
Then $\mathcal{C}_{i} = (\mathcal{X}_{i},\mathcal{K}_{i})$, $1
\leq i \leq 3$, are three supercharacter theories of $A_{10}$
different from $m(A_{10})$ and $M(A_{10})$.

\item We claim that the alternating group $A_{14}$ has at least $5$ supercharacter
theories. These are $m(A_{14})$, $M(A_{14})$, $\mathcal{C}_{1} =
(\mathcal{X}_{1}, \mathcal{K}_{1})$,  $\mathcal{C}_{2} =
(\mathcal{X}_{2},\mathcal{K}_{2})$ and
 $\mathcal{C}_{3} = \mathcal{C}_{1} \vee \mathcal{C}_{2}$ such that
\begin{small}
\begin{eqnarray*}
\mathcal{K}_{1} &=& \left\lbrace \{1\}, \{x_{i}^{A_{14}}\} (2 \leq i \leq 70), \{x_{71}^{A_{14}}, x_{72}^{A_{14}}\}\right\rbrace, \\
\mathcal{X}_{1} &=& \left\lbrace \{1\}, \{\chi_{i}\} (2 \leq i
\leq 19), \{\chi_{20}, \chi_{21}\}, \{\chi_{i}\} (22 \leq i \leq
72) \right\rbrace,\\
\mathcal{K}_{2} &=& \left\lbrace \{1\}, \{x_{i}^{A_{14}}\} (2 \leq i \leq 67), \{x_{68}^{A_{14}}, x_{69}^{A_{14}}\}, \{x_{i}^{A_{14}}\} (70 \leq i \leq 72)\right\rbrace, \\
\mathcal{X}_{2} &=& \left\lbrace \{1\}, \{\chi_{i}\} (2 \leq i
\leq 56), \{\chi_{57}, \chi_{58}\}, \{\chi_{i}\} (59 \leq i \leq
72) \right\rbrace.
\end{eqnarray*}

\end{small}\end{enumerate}

\item \textit{The alternating group $A_{n}$ has exactly one pair of non-real valued
irreducible character.} In this case, we will prove $ n \in \{7,
9, 11, 15, 18, 19, 23\}$. Since the number of non-real irreducible
characters is equal to the number of conjugacy classes that are
not preserved by the inversion mapping, we need those $n$ for
which $A_n$ has exactly two such conjugacy classes. Suppose $x
\in A_n$. By \cite[Proposition 12.17(2)]{110}, $x^{S_n} = x^{A_n}
\cup (1 2)x(1 2)^{A_n}$ if and only if $C_{S_n}(x) = C_{A_n}(x)$.
By Theorem \ref{4.1}(2), the last one is satisfied if and only if
$\dfrac{n - r}{2}$ is odd, where $r$ is the number of cycles in
decomposition of $x$. Our aim is to find all natural numbers $n$
such that $A_n$ has exactly one non-real class. Our main proof
will consider the following separate cases:
\begin{enumerate}
\item $n \equiv 0 ~(mod ~4)$. If $n \geq 8$ then $[1, n - 1]$ and $[3, n - 3]$ are two partitions with given
properties. So, $n = 4$ which contradicts by our main assumption
that $n \geq 5$.

\item $n \equiv 1~ (mod~ 4)$. If $ n \geq 13 $ then $ [1, 3, n - 4] $ and $ [1, 5, n - 6]
$ are two partitions with this property that all parts are odd and
$\frac{n-3}{2}$ ($r = 3$) is odd. Thus $n = 5$ or $9$. By Theorem
\ref{4.1}(1), the alternating group $A_5$ does not have non-real
class and the alternating group $A_9$ has exactly a unique pair of
non-real class.

\item $n \equiv 2~ (mod~ 4)$. If $ n \geq 22 $ then $ [1, 3, 5, n - 9] $ and $ [1, 3, 7, n - 11]
$ are  two partitions with given properties and so $n = 6, 10,
14$ or $18$. By Theorem \ref{4.1}(1), the alternating groups
$A_6, A_{10}$ and $A_{14}$ don't have non-real conjugacy class,
but the alternating group $A_{18}$ has a unique pair of conjugate
non-real characters.

\item $n \equiv 3~ (mod~ 4)$. If $ n \geq 31 $ then $ [1, 3, 5, 7, n - 16] $ and $ [1, 3, 5, 9, n - 18]
$ are two partitions with this property that all parts are odd and
$\frac{n-5}{2}$ ($r = 5$) is odd. Thus $n = 7, 11, 15, 19, 23$ or
$27$. Since $[27]$ and $[1,5,7,9,11]$ are two partitions with
mentioned properties, the case of $n = 27$ cannot be happened.
Other cases are solution of our problem.
\end{enumerate}

Hence, the alternating group $A_n$, $n \geq 5$, has a unique
non-real conjugacy class if and only if $n \in \{ 7, 9, 11, 15,
18, 19, 23\}$. A simple calculations by GAP shows that the
alternating group $A_7$ has exactly three supercharacter theories
$m(A_{7})$, $M(A_{7})$ and $\mathcal{C} = (\mathcal{X},
\mathcal{K})$, where
\begin{small}
\begin{eqnarray*}
\mathcal{K} &=& \left\lbrace \{1\}, \{x_{i}^{A_{7}}\} (2 \leq i \leq 7), \{x_{8}^{A_{7}}, x_{9}^{A_{7}}\}\right\rbrace, \\
\mathcal{X} &=& \left\lbrace \{1\}, \chi_{2}, \{\chi_{3},
\chi_{4}\}, \{\chi_{i}\} (5 \leq i \leq 9) \right\rbrace.
\end{eqnarray*}
\end{small}
The alternating group $A_{9}$ has at least $5$ supercharacter
theories $m(A_{9})$, $M(A_{9})$, $\mathcal{C}_{1} =
(\mathcal{X}_{1}, \mathcal{K}_{1})$, $\mathcal{C}_{2} =
(\mathcal{X}_{2}, \mathcal{K}_{2})$ and $\mathcal{C}_{3} =
\mathcal{C}_{1} \vee \mathcal{C}_{2}$. The partitions of
conjugacy classes and irreducible characters of $\mathcal{C}_1$
and $\mathcal{C}_2$ are defined as follows:
\begin{small}
\begin{eqnarray*}
\mathcal{K}_{1} &=& \left\lbrace \{1\}, \{x_{i}^{A_{9}}\} (2 \leq i \leq 12), \{x_{13}^{A_{9}},x_{14}^{A_{9}}\}, \{x_{i}^{A_{9}}\} (15 \leq i \leq 18) \right\rbrace, \\
\mathcal{X}_{1} &=& \left\lbrace \{1\}, \{\chi_{2}\}, \{\chi_{3},
\chi_{4}\}, \{\chi_{i}\} (5 \leq i \leq 18)\right\rbrace, \\
\mathcal{K}_{2} &=& \left\lbrace \{1\}, \{x_{i}^{A_{9}}\} (2 \leq i \leq 16), \{x_{17}^{A_{9}},x_{18}^{A_{9}}\} \right\rbrace, \\
\mathcal{X}_{2} &=& \left\lbrace \{1\}, \{\chi_{i}\} (2 \leq i
\leq 6), \{\chi_{7}, \chi_{8}\}, \{\chi_{i}\} (9 \leq i \leq
18)\right\rbrace.
\end{eqnarray*}
\end{small}
The alternating groups $A_{11}, A_{15}, A_{18}$, $A_{19}$ and
$A_{23}$ has at least one non-real irreducible character  and one
non-rational irreducible real character. So, by Lemma \ref{11} it
has at least five supercharacter theories $m$, $M$,
$\mathcal{C}_1$, $\mathcal{C}_2$ and $\mathcal{C}_3 =
\mathcal{C}_1 \vee \mathcal{C}_2$, proving this case.

\item \textit{The alternating group $A_{n}$ has at least two pairs of non-real valued
irreducible characters.} We first assume that $n > 24$. If $A =
\{ \chi(x) \ | \ \chi \in Irr(G)\}$ and $\mathbb{Q}(A)$ denotes
the filed generated by $\mathbb{Q}$ and $A$ then by
\cite[Theorem]{112}, the character table of $A_n$ has both
irrational and non-real character values. On the other hand, by
\cite[Theorem 2.5.13]{111}, each row or column of the character
table of $A_n$ contains at most one pair of irrational numbers.
Now by Lemma \ref{11}, we have at least five supercharacter
theories, as required.

Next we assume that $n \leq 24$. Set $\Gamma_1 = \{ 5, 6, 7, 9,
10, 11, 14, 15, 18, 19, 23\}$ and $\Gamma_2 = \{ 8, 12, 13, 16,
17, 20, 21, 22, 24\}$. If $n \in \Gamma_1$ then the number of
supercharacter theories of $A_n$ are investigated in Cases (1)
and (2). So, we have to prove that $s(A_n) \geq 5$, when $n \in
\Gamma_2$. By an easy calculation with GAP, one can see that if
$n \in \Gamma_2$ then $A_n$ has at least two pairs of non-real
valued irreducible characters and by Lemma \ref{11}, $s(A_n) \geq
5$.
\end{enumerate}
Hence the result.
\end{proof}

In the end of this paper we prove that the simple Suzuki group
$Sz(q), q = 2^{2n+1}$ has at least six super character theories.

\begin{thm}
The Suzuki group $Sz(q)$ has at least $6$ supercharacter theories
as: $m(Sz(q))$, $ M(Sz(q)) $, $\mathcal{C}_{1} = (\mathcal{X}_{1},
\mathcal{K}_{1})$, $\mathcal{C}_{2} = (\mathcal{X}_{2},
\mathcal{K}_{2})$, $\mathcal{C}_{3} = (\mathcal{X}_{3},
\mathcal{K}_{3})$ and $\mathcal{C}_{4} = (\mathcal{X}_{4},
\mathcal{K}_{4})$, where
\begin{small}
\begin{eqnarray*}
\mathcal{K}_{1} &=& \{\{1\},\{\rho, \rho^{-1}\}, Con(G(q)) - \{\rho, \rho^{-1}\}\},\\
\mathcal{X}_{1} &=& \{\{1\},\{W_{1}, W_{2}\}, Irr(G(q))-\{W_{1}, W_{2}\}\},\\
\mathcal{K}_{2} &=& \{\{1\},\{\pi_{0}\}, Con(G(q)) - \{\pi_{0}\}\},\\
\mathcal{X}_{2} &=& \{\{1\},\{X_{i}\}, Irr(G(q))-\{X_{i}\}\},\\
\mathcal{K}_{3} &=& \{\{1\},\{\pi_{1}\}, Con(G(q)) - \{\pi_{1}\}\},\\
\mathcal{X}_{3} &=& \{\{1\},\{Y_{j}\}, Irr(G(q))-\{Y_{j}\}\},\\
\mathcal{K}_{4} &=& \{\{1\},\{\pi_{2}\}, Con(G(q)) - \{\pi_{2}\}\},\\
\mathcal{X}_{4} &=& \{\{1\},\{Z_{k}\}, Irr(G(q))-\{Z_{k}\}\}.
\end{eqnarray*}

\end{small}\end{thm}

\begin{proof}

The schematic form of the character table of $Sz(q), q =
2^{2n+1}$, is shown in Table \ref{abc}, see  \cite{15,16} for
details.

\begin{table}[htp]
\caption{The Schematic Form of the Character Table of
$Sz(q)$.}\label{abc}
\begin{center}
\begin{tabular}{ccccccc}
Irreducible Characters & Degrees & $\#$Irreducible Characters\\
\hline
$X$ & $q^{2}$ & $1$ \\
$X_{i}$ & $q^{2} + 1$ & $q/2 - 1$  \\
$Y_{j}$ & $(q - r + 1)(q - 1)$ & $(q + r)/4$\\
$Z_{k}$ & $(q + r + 1)(q - 1)$ & $(q - r)/4$\\
$W_{l}$ & $r(q - 1)/2$ & $2$\\ \hline
\end{tabular}
\end{center}
\end{table}

In this table, the first column designates the characters, the
second column indicates the degrees, and the last one is the
number of characters of each degree.

Suppose $2q = r^{2}$. The Suzuki group $Sz(q)$ contains cyclic
groups of order $ q - 1 $, $q + r + 1$ and $ q - r + 1 $. These
subgroups are denoted by $A_{0}$, $ A_{1} $ and $ A_{2} $,
respectively. We also assume that  $ \pi_{i} $ is a typical
non-identity element of $ A_{i} $, $i = 0, 1, 2 $,  $ \sigma =
(0, 1)$ and $ \rho = (1, 0) $.

Let $ \varepsilon_{0} $, $ \varepsilon_{1} $ and  $
\varepsilon_{2} $ be a $ (q - 1)^{th} $, a primitive $ (q + r +
1)^{th} $  and a $ (q - r + 1)^{th} $ root of unity. We also
assume that  $ \xi_{0} $, $ \xi_{1} $ and $ \xi_{2} $ are
generators of $ A_{0} $, $ A_{1} $ and $ A_{2} $, respectively.
Define $ \varepsilon_{0}^{i} $, $ \varepsilon_{1}^{i} $ and $
\varepsilon_{2}^{i} $ as follows:

\begin{eqnarray*}
\varepsilon_{0}^{i}(\xi_{0}^{j}) &=& \varepsilon_{0}^{ij} +
\varepsilon_{0}^{-ij};~~~~(~i = 1, \ldots, q/2 - 1),\\
\varepsilon_{1}^{i}(\xi_{1}^{k}) &=& \varepsilon_{1}^{ik} +
\varepsilon_{1}^{ikq} + \varepsilon_{1}^{-ik} +
\varepsilon_{1}^{-ikq};~~~~(~i = 1, \ldots, q + r),\\
\varepsilon_{2}^{i}(\xi_{2}^{k}) &=& \varepsilon_{2}^{ik} +
\varepsilon_{2}^{ikq} + \varepsilon_{2}^{-ik} +
\varepsilon_{2}^{-ikq}.
\end{eqnarray*}
The functions $ \varepsilon_{0}^{i} $, $ \varepsilon_{1}^{i} $
and $ \varepsilon_{2}^{i} $ are characters of $ A_{0} $, $ A_{1}
$ and $ A_{2} $, respectively. Following Suzuki \cite{15}, the
character table of $Sz(q)$, is computed in Table \ref{bcd}.

\begin{table}[htp]
\caption{The character table of $Sz(q)$.} \label{bcd}
\centering
\begin{tabular}{ccccccc}
$$ & $1$ & $\sigma$ & $\rho, \rho^{-1}$ & $\pi_{0}$ & $\pi_{1}$ & $\pi_{2}$ \\
$X$ & $q^{2}$ & $0$ & $0$ & $1$ & $-1$ & $-1$ \\
$X_{i}$ & $q^{2} + 1$ & $1$ & $1$ & $\varepsilon_{0}^{i}(\pi_{0})$ & $0$ & $0$ \\
$Y_{j}$ & $(q - r + 1)(q - 1)$ & $r - 1$ & $-1$ & $0$ & $-\varepsilon_{1}^{j}(\pi_{0})$ & $0$ \\
$Z_{k}$ & $(q + r + 1)(q - 1)$ & $-r - 1$ & $-1$ & $0$ & $0$ & $-\varepsilon_{2}^{k}(\pi_{0})$ \\
$W_{l}$ & $r(q - 1)/2$ & $-r/2$ & $\pm r\sqrt{-1}/2$ & $0$ & $1$ & $-1$ \\
\end{tabular}
\end{table}

We now apply Lemma \ref{11} to construct four supercharacter
theories given the statement of this theorem. By considering
$m(Sz(q))$ and $M(Sz(q))$, it can be proved that $s(Sz(q)) \geq 6$.
\end{proof}

\noindent{\bf Acknowledgement.} The authors are indebted to
professors Marston Conder, Geoffrey R. Robinson and Jeremy
Rickard for some critical discussion through Group Pub Forum and
MathOverFlow on Theorem 3.3. The research of the authors are
partially supported by the University of Kashan under grant no
572760/1.

\end{document}